\documentclass[reqno,a4paper,draft]{amsart}

\usepackage{enumitem}
\setenumerate{label=\textnormal{(\arabic*)}}

\usepackage{amsmath,amssymb,dsfont,verbatim,bm,array,mathtools}
\usepackage[latin1]{inputenc}
\usepackage{booktabs}
\usepackage[raggedright]{titlesec}
\usepackage{mathtools}

\titleformat{\chapter}[display]
{\normalfont\huge\bfseries}{\chaptertitlename\\thechapter}{20pt}{\Huge}
\titleformat{\section}
{\normalfont\Large\bfseries\center}{\thesection}{1em}{}
\titleformat{\subsection}
{\normalfont\large\bfseries}{\thesubsection}{1em}{}
\titleformat{\subsubsection}[runin]
{\normalfont\normalsize\bfseries}{\thesubsubsection}{1em}{}
\titleformat{\paragraph}[runin]
{\normalfont\normalsize\bfseries}{\theparagraph}{1em}{}
\titleformat{\subparagraph}[runin]
{\normalfont\normalsize\bfseries}{\thesubparagraph}{1em}{}
\titlespacing*{\chapter} {0pt}{50pt}{40pt}
\titlespacing*{\section} {0pt}{3.5ex plus 1ex minus .2ex}{2.3ex plus .2ex}
\titlespacing*{\subsection} {0pt}{3.25ex plus 1ex minus .2ex}{1.5ex plus .2ex}
\titlespacing*{\subsubsection}{0pt}{3.25ex plus 1ex minus .2ex}{1.5ex plus .2ex}
\titlespacing*{\paragraph} {0pt}{3.25ex plus 1ex minus .2ex}{1em}
\titlespacing*{\subparagraph} {\parindent}{3.25ex plus 1ex minus .2ex}{1em}

\input xypic
\xyoption{all}


\newtheorem{theorem}{Theorem}[section]

\theoremstyle{definition}

\theoremstyle{remark}

\DeclareMathOperator{\Jac}{Jac}

\begin{document}
\title{The two-dimensional Jacobian Conjecture and unique factorization}

\author{Vered Moskowicz}
\address{Department of Mathematics, Bar-Ilan University, Ramat-Gan 52900, Israel.}
\email{vered.moskowicz@gmail.com}

\begin{abstract}
The two-dimensional Jacobian Conjecture says that a 
$\mathbb{C}$-algebra endomorphism 
$F:\mathbb{C}[x,y] \to \mathbb{C}[x,y]$ 
that has an invertible Jacobian is an automorphism.

We show that if a 
$\mathbb{C}$-algebra endomorphism 
$F:\mathbb{C}[x,y] \to \mathbb{C}[x,y]$ 
has an invertible Jacobian 
and if 
$v \in \mathbb{C}[F(x),F(y),x]$
is a product of prime elements of 
$\mathbb{C}[F(x),F(y),x]$,
then $F$ is an automorphism,
where $v$ is such that
$y = u/v$, where 
$u \in \mathbb{C}[F(x),F(y),x]$.
\end{abstract}

\maketitle

\section{Introduction}
The famous $n$-dimensional Jacobian Conjecture,
raised by O.H. Keller ~\cite{keller} in 1939,
says that a 
$\mathbb{C}$-algebra endomorphism 
$F:\mathbb{C}[x_1,\ldots,x_n] \to \mathbb{C}[x_1,\ldots,x_n]$ 
that has an invertible Jacobian is an automorphism.

A theorem of E. Formanek ~\cite[Theorem 2]{formanek observations}
(which generalizes a theorem of T.T. Moh ~\cite[page 151]{moh}),
says the following:
Let 
$F:\mathbb{C}[x_1,\ldots,x_n] \to \mathbb{C}[x_1,\ldots,x_n]$
be a $\mathbb{C}$-algebra endomorphism 
that satisfies
$\Jac(F(x_1),\ldots,F(x_n)) \in \mathbb{C}^*$.
Then
$\mathbb{C}(F(x_1),\ldots,F(x_n),x_1,\ldots,x_{n-1})=\mathbb{C}(x_1,\ldots,x_n)$.

We will show that if a 
$\mathbb{C}$-algebra endomorphism 
$F:\mathbb{C}[x,y] \to \mathbb{C}[x,y]$ 
has an invertible Jacobian 
and if 
$v \in \mathbb{C}[F(x),F(y),x]$
is a product of prime elements of 
$\mathbb{C}[F(x),F(y),x]$,
then $F$ is an automorphism,
where $v$ is such that
$y = u/v$, where 
$u \in \mathbb{C}[F(x),F(y),x]$.
We do not know if an analogous result for $n \geq 3$ exists.
\section{Our result}
For results about factorization in
integral domains, see, for example,
~\cite[Chapter 15]{pete}. 
We begin with the following observation:
\begin{theorem}\label{1}
Let
$F:\mathbb{C}[x,y] \to \mathbb{C}[x,y]$
be a $\mathbb{C}$-algebra endomorphism
having an invertible Jacobian.
If
$\mathbb{C}[F(x),F(y),x]$ is a UFD,
then $F$ is an automorphism.
\end{theorem}

We will bring two proofs for Theorem \ref{1};
in both proofs we will use another theorem of Formanek ~\cite[Theorem 1]{formanek two notes},
see also ~\cite[Exercise 9, page 13]{essen}:
Let $k$ be a field of characteristic zero
and let
$F:k[x_1,\ldots,x_n] \to k[x_1,\ldots,x_n]$ 
be a $k$-algebra endomorphism  
that satisfies
$\Jac(F(x_1),\ldots,F(x_n)) \in k^*$.
Suppose that there is a polynomial 
$F_{n+1}$ in $k[x_1,\ldots,x_n]$ 
such that
$k[F(x_1),\ldots,F(x_n),F_{n+1}]=
k[x_1,\ldots,x_n]$.
Then 
$k[F(x_1),\ldots,F(x_n)]=k[x_1,\ldots,x_n]$
(= $F$ is a $k$-algebra automorphism of $k[x_1,\ldots,x_n]$).

\begin{proof}
\textbf{First, almost immediate, proof:}
Write 
$F(x)= a_my^m + a_{m-1}y^{m-1} + \ldots + a_1y + a_0$, 
where $a_j \in \mathbb{C}[x]$.
We can assume that $F(x)$ is monic in $y$ (and also $F(y)$ is monic in $y$, 
but we will not use this additional fact).
Indeed, multiply $F$ by an automorphism $g$,
$g(x)= x+y^N$, $g(y)= y$
for appropriate $N$, and get that 
$(gF)(x)= y^m + b_{m-1}y^{m-1} + \ldots + b_1y + b_0$, 
where $b_j \in \mathbb{C}[x]$.

Now assume that $F(x)$ is monic in $y$.
Hence, $y$ is integral over $\mathbb{C}[F(x),x]$, 
and then $y$ is integral over $\mathbb{C}[F(x),F(y),x]$,
so, 
$\mathbb{C}[F(x),F(y),x] \subseteq \mathbb{C}[F(x),F(y),x][y] = \mathbb{C}[x,y]$ 
is an integral extension.
By Formanek's theorem ~\cite[Theorem 2]{formanek observations}, we have
$\mathbb{C}(F(x),F(y),x) = \mathbb{C}(x,y)$.

By assumption, $\mathbb{C}[F(x),F(y),x]$ is a UFD, hence integrally closed (in its field of fractions)
~\cite[Theorem 15.14]{pete}, therefore,
$\mathbb{C}[F(x),F(y),x] = \mathbb{C}[x,y]$.

Finally, Formanek's theorem ~\cite[Theorem 1]{formanek two notes} implies that
$\mathbb{C}[F(x),F(y)] = \mathbb{C} [x,y]$.

\textbf{Second proof:}
The second proof can be found in ~\cite[Theorem 3.1]{vered sep}.

Recall Adjamagbo's transfer theorem ~\cite[Theorem 1.7]{adja}: 
Given commutative rings $A \subseteq B \subseteq C$ such that: $A$ is normal and Noetherian,
$B$ is isomorphic to $A[T]/hA[T]$, where $A[T]$ is the $A$-algebra of polynomials generated by one indeterminate 
$T$ and $h \in A[T]-A$, $C$ an affine $B$-algebra, $C$ is separable over $A$,
$C^*=A^*$ and the prime spectrum of $C$ is connected.
Then the following conditions are equivalent:\begin{itemize}
\item [(1)] $B$ is normal.
\item [(2)] $C$ is flat over $B$.
\item [(3)] $B$ is separable over $A$.
\item [(4)] $B$ is \'{e}tale (=unramified and flat) over $A$.
\end{itemize}
It is not difficult to check that 
$A=\mathbb{C}[F(x),F(y)]$,
$B=\mathbb{C}[F(x),F(y)][x]$
and 
$C=\mathbb{C}[x,y]$
satisfy the assumptions in Adjamagbo's transfer theorem.
Therefore, if we show that one of conditions 
$(1)-(4)$ is satisfied,
then also all the other conditions are satisfied.

By assumption $\mathbb{C}[F(x),F(y),x]$ is a UFD,
so condition $(1)$ is satisfied, hence,
in particular, condition $(4)$ is satisfied.

Recall Bass's theorem ~\cite[Proposition 1.1]{bass}: Let $k$ be an algebraically closed field of
characteristic zero. Assume that $k[x_1,x_2] \subseteq B$ 
is an affine integral domain over $k$ which is an unramified extension of $k[x_1,x_2]$.
Assume also that $B=k[x_1,x_2][b]$ for some $b \in B$. If $B^*=k^*$ then $B=k[x_1,x_2]$.

Now apply Bass's theorem to 
$\mathbb{C}[F(x),F(y)] \subseteq \mathbb{C}[F(x),F(y)][x]$
and get that
$\mathbb{C}[F(x),F(y)] = \mathbb{C}[F(x),F(y)][x]$.

Then we clearly have, 
$\mathbb{C}[F(x),F(y),y] = \mathbb{C}[x,y]$,
hence Formanek's theorem ~\cite[Theorem 1]{formanek two notes} implies that
$\mathbb{C}[F(x),F(y)] = \mathbb{C}[x,y]$.
\end{proof}

We wish to find a weaker assumption than 
$\mathbb{C}[F(x),F(y),x]$ being a UFD.

A first option is to assume that 
$\mathbb{C}[F(x),F(y),x]$ is a GCD-domain.
However, it is than immediate that 
$\mathbb{C}[F(x),F(y),x]$ is a UFD,
see ~\cite[Corollary 15.13]{pete}.

A second option is to assume that 
a specific element $v \in \mathbb{C}[F(x),F(y),x]-0$ 
has a factorization into prime elements of 
$\mathbb{C}[F(x),F(y),x]$
(if such a factorization exists,
then it is necessarily unique up to units).

Of course, every non-zero element of
$\mathbb{C}[F(x),F(y),x]$ has (at least)
one factorization into irreducible elements
of $\mathbb{C}[F(x),F(y),x]$, 
see ~\cite[Proposition 15.3]{pete},
but the problem is that those irreducible elements
need not be prime elements,
and for our proof of Theorem \ref{2}
we need a factorization of $v$ into primes.

Now, in the introduction we recalled 
Formanek's theorem ~\cite[Theorem 2]{formanek observations}.
When $n=2$ it says that if
$F:\mathbb{C}[x,y] \to \mathbb{C}[x,y]$
is a $\mathbb{C}$-algebra endomorphism 
having an invertible Jacobian,
then
$\mathbb{C}(F(x),F(y),x)=\mathbb{C}(x,y)$.
Therefore, 
for a given 
$\mathbb{C}$-algebra endomorphism
$F:\mathbb{C}[x,y] \to \mathbb{C}[x,y]$
having an invertible Jacobian,
there exist 
$u,v \in \mathbb{C}[F(x),F(y),x]-0$
such that
$y=u/v$.
(Actually, $v$ can be taken from
$\mathbb{C}[F(x),F(y)]-0$,
see ~\cite{vered},
but we will not use this fact here).

\begin{theorem}\label{2}
Let
$F:\mathbb{C}[x,y] \to \mathbb{C}[x,y]$
be a $\mathbb{C}$-algebra endomorphism
having an invertible Jacobian.
Assume that
$v \in \mathbb{C}[F(x),F(y),x]-0$
as above
is a product of prime elements of 
$\mathbb{C}[F(x),F(y),x]$.
Then $F$ is an automorphism.
\end{theorem}

\begin{proof}
Write 
$v=v_1 \cdots v_m$, where
$v_j \in \mathbb{C}[F(x),F(y),x]$
are prime elements of $\mathbb{C}[F(x),F(y),x]$.
Let $S_0$ be the multiplicative set $\{v^i \}$,
and let $S$ be the saturation of $S$,
see ~\cite[page 127]{pete};
this $S$ is \textit{primal} ~\cite[page 251]{pete}
(=generated by units and by prime elements).
We claim that the localization 
$S^{-1}\mathbb{C}[F(x),F(y),x]$ is a UFD;
indeed,
$S^{-1}\mathbb{C}[F(x),F(y),x]= 
S^{-1}\mathbb{C}[x,y]$,
and
$S^{-1}\mathbb{C}[x,y]$
is a UFD as a localization of the UFD
$\mathbb{C}[x,y]$
(see ~\cite[Theorem 15.36]{pete}).
Now by a theorem of Nagata ~\cite[Theorem 15.39]{pete},
we get that $\mathbb{C}[F(x),F(y),x]$ is a UFD.
Therefore, theorem \ref{1} implies that 
$F$ is an automorphism.

\end{proof}

The proof of Theorem \ref{2} shows that,
under the assumptions of Theorem \ref{2},
$\mathbb{C}[F(x),F(y),x]$ is a UFD,
so actually our weaker condition is not weaker but equivalent.

\bibliographystyle{plain}

\end{document}